\documentclass[a4paper, 11pt]{article}
\usepackage{amsmath,amssymb,esint,amscd,xspace,fancyhdr,color,authblk,srcltx,fontenc,bbm,amsxtra,latexsym}
\usepackage{framed}
\usepackage{hyperref}

\newtheorem{theorem}{Theorem}[section]

\newtheorem{corollary}[theorem]{Corollary}
\newtheorem{definition}[theorem]{Definition}
\newtheorem{lemma}[theorem]{Lemma}

\newenvironment{proof}[1][Proof]{\textbf{#1.} }{\hfill\rule{0.5em}{0.5em}}
{\catcode`\@=11\global\let\AddToReset=\@addtoreset
	
	\setlength{\oddsidemargin}{-0.05in}
	\setlength{\evensidemargin}{-0.05in}
	\setlength{\textwidth}{14.5cm}
	\textheight=22.15cm
	\voffset=-1truecm
	\hoffset=+1.1truecm
	
\title{Applications of optimal error bounds for some generalized two-step iterative processes in Banach spaces}
	
\author{Tan-Phuc Nguyen, Thai-Hung Nguyen\footnote{Corresponding author.},  Tien-Khai Nguyen, Cong-Duy-Nguyen Nguyen, Trung-Hieu Huynh\thanks{Department of Mathematics, Ho Chi Minh City University of Education, Ho Chi Minh city, Vietnam; Emails: \texttt{thaihungspt2003@gmail.com; tanphucb11@gmail.com; ngtienkhai59@gmail.com; nguyennguyen04012004@gmail.com; huynhhieu2004@gmail.com}}}
	
	\date{\today} 
	
\begin{document}
		\large
		\maketitle
		\begin{abstract}

In a recent paper~\cite{paper2}, we proposed the concept of optimal error bounds for an iterative process, which allows us to obtain the convergence result of the iterative sequence to the common fixed point of the nonexpansive mappings in Banach spaces. Moreover, we also achieve the comparison results between different iterative processes via optimal error bounds. In this paper, we continue to determine optimal error bounds for more general iterative processes which were studied by many authors, such as in~\cite{DungHieu} and references therein. From there, the convergence results are obtained and the convergence rates of these iterative processes are determined under some sufficient conditions on sequences of parameters.

			\medskip
			
			\medskip
			\noindent {\emph{Keywords:} Fixed point theory; $(\kappa,\alpha)$-nonexpansive mappings; Generalized iterative processes; Optimal error bounds; Convergence rate.} 
		\end{abstract}
		\tableofcontents
		
		
\section{Introduction}\label{Sec1}

Many mathematical models can lead to the problem of finding fixed points for a mapping. Finding such points is crucial in various areas of mathematics, as they often represent stable states or equilibrium solutions in dynamic systems. Fixed point problems also have applications across many fields both within and beyond mathematics, such as approximation theory, optimization theory, game theory, finance, machine learning, \dots One of the most fundamental results of this theory is the Banach  fixed point theorem, which asserts the existence of the unique fixed point of a contraction mapping $T:\mathbb{X}\to\mathbb R$, where $(\mathbb X,\lVert\cdot\rVert)$ is a Banach space in $\mathbb R$. Sometimes, finding the exact fixed point is not feasible, so methods to approximate it are sought. This fixed point can be approximated by using the following Picard iteration
  \begin{align*}\label{Picard}
          x_0\in \mathbb X, \ 
        x_{n+1}=T(x_n), \quad n\in \mathbb N.\tag{1.1}
     \end{align*}
    There are many ideas for constructing the iteration such that the limit of the iteration is the fixed point of the mapping under consideration. These methods aim to enhance the applicability while simultaneously increasing the convergence rate of the iteration to the fixed point. For example, in the case where $T$ is not a contraction mapping, the convergence of the iterative sequence~\eqref{Picard} to the fixed point of $T$ is not guaranteed. One of the earliest methods to overcome this difficulty is the iterative sequence proposed by Mann in~\cite{Mann}, defined by
    \begin{align*}\label{Mann}
            x_0\in \mathbb X,\
            x_{n+1}=(1-a_n)x_n+a_nT(x_n),\quad n\in \mathbb N.\tag{1.2}
         \end{align*}
    where $(a_n)_{n \in \mathbb{N}}\subset [0,1]$ is a  parameter sequence satisfying 
    $$\lim\limits_{n\to\infty} a_n=0 \ \mbox{ and } \ \displaystyle\sum\limits_{n \in \mathbb{N}}{a_n}=+\infty.$$ 
    If the sequence $(a_n)$ is a constant sequence, that is $a_n = \theta \in [0, 1]$ for all $n \in \mathbb N$, the Mann iteration~\eqref{Mann} is refered to as a Krasnoselskij iteration. With appropriate assumptions about the sequence $(a_n)$, we can prove the convergence of the Mann iteration~\eqref{Mann} to the fixed point $x^*$ of the mapping $T$.  A substantial amount of research has been devoted to studying the convergence of Mann iterations for various types of nonlinear mappings across different function spaces. However, it has been shown that the Mann iteration converges to the fixed point more slowly than the Picard iteration under many different assumptions regarding the mapping $T$ (cf.~\cite{Popescu2007}). Consequently, many mathematicians have proposed two or three-step iterations to improve the existing iterations (cf.~\cite{D70, Halpern, Ishikawa, KX2005}). One of the typical two-step iterations that has been proposed is the Ishikawa iteration in~\cite{Ishikawa}
    \begin{align*}\label{1.3}
			\begin{cases} 
   x_0\in \mathbb X,\\
   y_n = (1-a_n)x_n + a_nT(x_n), \\ x_{n+1} = (1-b_n)x_n + b_nT(y_n), \quad n \in \mathbb N,\tag{1.3}
			\end{cases}
		\end{align*}
where $(a_n),(b_n)$ are two real sequences in $[0,1]$. Many different ideas have been proposed
to improve the Ishikawa iterative sequence. To be more precise, in~\cite{TT1998,YC2007}, an iterative process was defined as follows  
\begin{align*}\label{1.4}
    \begin{cases} 
   x_0\in \mathbb X,\\
   y_n = (1-a_n)x_n + a_nT_1(x_n), \\ 
   x_{n+1} = (1-b_n)x_n + b_nT_2(y_n), \quad n \in \mathbb N,\tag{1.4}
\end{cases}
\end{align*}
where $T_1,T_2$ are non-expansive mappings. It can be seen that the iteration sequence~\eqref{1.3} is a specific case of~\eqref{1.4} when $T_1=T_2=T$. Many authors developed conditions for the parameter sequences $(a_n),(b_n)$ in \eqref{1.4} to receive interesting convergence results to the common fixed point of $T_1,T_2$. Moreover, over the years, many other iterative sequences also have been studied and improved to achieve small errors and higher rates of convergence to the common fixed point (cf.~\cite{Berinde2008},\cite{CCK2003},\cite{KX2005},\cite{Noor}-\cite{SP2011},\cite{Xue}).\medskip

Recently, in~\cite{paper2} we proposed the concept of \textit{optimal error bounds} to investigate the convergence of multi-step iterative sequences to a common fixed point of two given non-expansive mappings $T_1$ and $T_2$. In detail, we used \textit{optimal error bounds} to construct necessary and sufficient conditions for parameter sequences in the corresponding two-step Ishikawa iterative sequences to converge, and we compared the convergence rates between them. This approach is entirely new compared to previous research, based on the idea that there is no better bound without additional assumptions regarding the mappings, even in the general case. We find that this iterative method can be extended to multi-step iterative sequences converging to a common fixed point of multiple non-expansive mappings. Inspired by this, our purpose in this paper is using \textit{optimal error bounds} to analyze the convergence and rate of convergence of more general iterative sequences in~\cite{DungHieu} and some other references. Specifically, we also establish necessary and sufficient conditions of \textit{optimal error bounds} for some generalized two-step iterative sequences to converge to a common fixed point of two or more non-expansive mappings, and then compare the convergence rates between the considered sequences. Thereby, we confirm the potential of the concept \textit{optimal error bounds} in studying the convergence and convergence rates of the iterative methods to the fixed point of some given non-expansive mappings.\medskip

The rest of the paper is organized as follows. In the next section, we recall the definition of optimal error bounds for the iterative process and the representations of some iterative processes to be investigated. In Section~\ref{Sec3}, we will establish the optimal error bounds for two general iterative processes (IG) and (G). Convergence of iterative sequences is also presented in this section as the first application of optimal error bounds. The next application of optimal error bounds is the comparison result of convergence rates of iterative sequences. These results are presented in the last section. 

\newpage
\section{Some generalized two-step processes}\label{Sec2}

First, we introduce the definition of optimal error bounds for a generalized iterative sequence and the iterative sequences examined in this paper. It is worth noting that the optimal error bounds definition we will present below is a generalization of the one found in~\cite{paper2}. Throughout this article, we always assume $\mathbb{X}$ is a Banach space in $\mathbb{R}$ and $\Omega$ is a closed convex subset in  $\mathbb{X}$. 

\begin{definition}[$(\kappa,\alpha)$-nonexpansive mapping] Given $0 \le \kappa \le \alpha \le 1$, a mapping $T: \Omega \to \Omega$ is called $(\kappa,\alpha)$-nonexpansive if
\begin{align}\notag 
\kappa \|\xi-\eta\| \le \|T(\xi) - T(\eta)\| \le \alpha \|\xi-\eta\|, 
\end{align}
for every $\xi,\eta \in \Omega$. In this case, we write $T \in \mathcal{N}_{\kappa,\alpha}(\Omega)$. If $\alpha<1$, we say that $T$ is $(\kappa,\alpha)$-contraction, and write $T \in \mathcal{C}_{\kappa,\alpha}(\Omega)$. Moreover, if $\kappa=0$ then we simply write $T \in \mathcal{N}_{\alpha}(\Omega)$ or $T \in \mathcal{C}_{\alpha}(\Omega)$ provided $\alpha<1$.
\end{definition} 
With these notations, it is possible to check that 
\begin{align*}
\mathcal{C}_{\kappa,\alpha}(\Omega) \subset \mathcal{C}_{\alpha}(\Omega) \cap \mathcal{N}_{\kappa,\alpha}(\Omega),  \mbox{ and } \mathcal{C}_{\alpha}(\Omega) \cup \mathcal{N}_{\kappa,\alpha}(\Omega) \subset \mathcal{N}_{\alpha}(\Omega).
\end{align*}

\begin{definition}[Abstract iterative process]
Let $K \in \mathbb{N}^*$ and $2K$ parameters $\kappa_k,\alpha_k \in [0,1]$ for  $k = 1,2,...,K$. An abstract iterative process $(\mathcal{P})$ can be defined as follows: for nonexpansive mappings $T_k \in \mathcal{N}_{\kappa_k,\alpha_k}(\Omega)$ satisfying $T_k(x^*) = x^*$, we define
\begin{align}\tag{$\mathcal{P}$}
\label{P}
 x_0 \in \Omega, \ x_{n+1} = \mathcal{P}_{T_k}(x_n), \quad n \in \mathbb{N},
\end{align}
where $\mathcal{P}_{T_k}: \Omega \to \Omega$ and $\mathcal{P}_{T_k}(x^*)=x^*$. 
\end{definition}

It is easy to check that if $x_0 = x^*$ then $x_n = x^*$ for all $n \in \mathbb{N}$. For his reason, without loss of generality, we may assume that $x_0 \neq x^*$.

\begin{definition}[Optimal error bounds]
Assume that $(\mathcal{L}_n^\mathcal{P})_{n\in \mathbb{N}}$ and $(\mathcal{U}_n^\mathcal{P})_{n\in \mathbb{N}}$ are two non-negative sequences such that the following inequality
\begin{align}\label{upper-bound}
\mathcal{L}_n^\mathcal{P} \le \frac{\|x_{n+1} - x^*\|}{\|x_0 - x^*\|} \le \mathcal{U}_n^\mathcal{P} , \quad \forall n \in \mathbb{N},
\end{align}
holds for any nonexpansive mappings $T_k \in \mathcal{N}_{\kappa_k,\alpha_k}(\Omega)$ satisfying $T_k(x^*) = x^*$. We say that the sequence $(\mathcal{L}_n^\mathcal{P})_{n\in \mathbb{N}}$ ($(\mathcal{U}_n^\mathcal{P})_{n\in \mathbb{N}}$ respectively) is optimal lower error bound (optimal upper error bound respectively) of the iteration~\eqref{P} if the corresponding equality in~\eqref{upper-bound} holds for some $T_k$.
\end{definition}

In this paper, we will establish the optimal error bounds for some two-step models of the same type as $(\mathcal{P})$. More precisely, we consider four generalized models which are presented as below. We always assume that $n \in \mathbb{N}$ and $x_0 \in \Omega$ is the initial point of each iterative sequence. The improved Ishikawa iteration (I) is defined by:		
\begin{align}\tag{I}\label{xn-I}
\begin{cases} y_{n} = (1-a_n)x_n + a_nT_1(x_n), \\ x_{n+1} = (1-b_n)x_n + b_nT_2(y_n).\end{cases}
\end{align}
We remark that the iteration~\eqref{xn-I} can be rewritten as follows
\begin{align*}
x_{n+1} = (1-b_n)x_n + b_nT_2\big((1-a_n)x_n + a_nT_1(x_n)\big).
\end{align*}
Hence, it can be considered as the abstract iterative process~\eqref{P} associated to the mapping $\mathcal{P}_{T_1,T_2}: \Omega \to \Omega$ below
\begin{align*}
\mathcal{P}_{T_1,T_2}(\xi) = (1-b_n)\xi + b_nT_2\big((1-a_n)\xi + a_nT_1(\xi)\big), \quad \xi \in \Omega.
\end{align*}
In the same spirit, we also consider more general two-step iterations, such as: the modified Ishikawa iterations (IM):	
\begin{align}\tag{IM}\label{xn-IM}
\begin{cases} y_{n} = (1-a_n)x_n + a_nT_1(x_n), \\ x_{n+1} = (1-b_n)y_n + b_nT_2(y_n);\end{cases}
\end{align}
the generalized Ishikawa iteration (IG):
\begin{align}\tag{IG}\label{xn-IG}
\begin{cases} y_{n} = (1-a_n)x_n + a_nT_1(x_n), \\ x_{n+1} = (1-b_n)T_1(y_n) + b_nT_2(y_n); \end{cases}
\end{align}
and the generalized two-step iteration (G):
\begin{align}\tag{G}\label{xn-G}
\begin{cases} y_{n} = (1-a_n)S_1(x_n) + a_nT_1(x_n), \\ x_{n+1} = (1-b_n)S_2(y_n) + b_nT_2(y_n).\end{cases}	
\end{align}

\section{Convergence analysis} \label{Sec3}
In this section, we establish the optimal lower error bounds (OLEBs) and the optimal upper error bounds (OUEBs) of two iterative processes~\eqref{xn-IG} and~\eqref{xn-G}. By controlling the sequence of OEBs, we derive some proper conditions for the optimal error bounds to converge to 0, ensuring that these iterative processes converge to a common fixed point of some given nonexpansive mappings. Let us first recall two fundamental lemmas which are useful in our works.
\begin{lemma}\label{lem1}
For every real number $x > -1$, we have 
\begin{align}\label{ln-ineq}
\dfrac{x}{x+1} \le \ln(1+ x) \le x.
\end{align}
\end{lemma}
		
\begin{lemma}\label{lem2}
Assume that sequence $(a_n) \subset [0, 1]$ and bounded sequence $(u_n)$  satisfies 
$$1 - u_na_n > 0, \quad \mbox{ for all } n\in \mathbb{N}.$$ 
Then, two series $\displaystyle \sum_{n \in \mathbb{N}} \dfrac{a_n}{1- u_na_n}$ and $\displaystyle \sum_{n \in \mathbb{N}} a_n$ jointly converge or jointly diverge. In this case, we will denote by
\begin{align*}
\sum_{n \in \mathbb{N}} \dfrac{a_n}{1- u_na_n} \sim \sum_{n \in \mathbb{N}} a_n.
\end{align*}
\end{lemma}
The proof of Lemma~\ref{lem2} can be found in \cite[Lemma 2.7]{paper2}.
		
\subsection{Generalized Ishikawa iteration (IG)}
		
Let us consider four constants $\alpha_1, \alpha_2, \kappa_1, \kappa_2 \in [0, 1]$ such that $0 < \alpha_1 + \alpha_2 < 2$, $\kappa_1 \le \alpha_1$, $\kappa_2 \le \alpha_2$ and two parameter sequences $(a_n), (b_n) \subset [0, 1]$. For every two mappings $T_1 \in \text{Ne}_{\kappa_1, \alpha_1}(\Omega)$, $T_2 \in \text{Ne}_{\kappa_2, \alpha_2}(\Omega)$ with the common fixed point $x^* \in \mathcal{F}(T_1, T_2)$, the iteration $(x_n)$ defined by
		\begin{align}\tag{IG}\label{xn-IG}
			\begin{cases} x_0 \in \Omega,\\ y_{n} = (1-a_n)x_n + a_nT_1(x_n), \\ x_{n+1} = (1-b_n)T_1(y_n) + b_nT_2(y_n), \ n \in \mathbb{N}, \end{cases}
		\end{align}
		is called the generalized Ishikawa iteration (IG) of two-step iterative process.
		
		\begin{theorem}\label{theo-U-IG}
			The optimal upper bound of~\eqref{xn-IG} is determined by
			\begin{align}\label{Un-IG}
				\mathcal{U}_n^{IG} := \displaystyle\prod_{k=0}^{n}[1-(1 - \alpha_1) a_k][\alpha_1+(\alpha_2 - \alpha_1)b_k],
			\end{align}
			for all $n \in \mathbb{N}$. Moreover, the sequence $(\mathcal{U}^{IG}_n)$ converges to $0$ if and only if
			\begin{align}\label{iff-Un-IG} 
				(1-\alpha_1) \sum_{k \in \mathbb{N}} a_k + (1-\alpha_1)\sum_{k \in \mathbb{N}} (1-b_k) + (1-\alpha_2)\sum_{k \in \mathbb{N}} b_k = +\infty. 
			\end{align}
		\end{theorem}
		\begin{proof}
			Since $T_1 \in \text{Ne}_{\kappa_1, \alpha_1}(\Omega)$ and $T_1(x^*) = x^*$, so for all $n\ge 1$, we have
			\begin{align*}
				\|y_n-x^*\|&=\|(1-a_n)(x_n-x^*)+a_n(T_1(x_n)-T_1(x^*))\| \\
				& \le (1-a_n) \|x_n-x^*\|+a_n\|T_1(x_n)-T_1(x^*)\|\\
				& \le (1-a_n)\|x_n-x^*\|+ \alpha_1a_n\|x_n-x^*\|\\
				& = [1 - (1 - \alpha_1)a_n]\|x_n-x^*\|.
			\end{align*}
			Similarly, since $T_1 \in \text{Ne}_{\kappa_1, \alpha_1}(\Omega)$, $T_2 \in \text{Ne}_{\kappa_2, \alpha_2}(\Omega)$ and $T_1(x^*) = T_2(x^*) = x^*$, we deduce that
			\begin{align}\label{est-IG} 
				\|x_{n+1} - x^*\| \le [1 - (1-\alpha_1)a_n][\alpha_2+(\alpha_1-\alpha_2)b_n]\|x_n-x^*\|,
			\end{align}
			for all $n \in \mathbb{N}$. From~\eqref{est-IG}, it yields that 
			\begin{align}\label{est-Un} 
				\|x_{n+1} - x^*\| \le \mathcal{U}_n^{IG} \|x_0-x^*\|,
			\end{align}
			for all $n \in \mathbb{N}$, which $\mathcal{U}_n^{IG}$ is defined by~\eqref{Un-IG}. \\
			
			It is not difficult to check that with the open set  $\displaystyle\Omega=\left(0,\frac{1}{2}\right)$, $\alpha_1, \alpha_2 \in [0, 1]$ and two nonexpansive mappings $T_1,T_2$ defined by
			$$T_1(x)=\alpha_1x, \quad  T_2=\alpha_2x, \ x \in \Omega,$$
			then we have $x^*=0$ and $\|x_{n+1}-x^*\|=\mathcal{U}^{IG}_{n}\|x_0-x^*\|, \text{ for all } n \in \mathbb{N}.$ Therefore, we conclude that $\mathcal{U}^{IG}_{n}$ is the optimal upper bound of  iteration~\eqref{xn-IG}. \\
			
			Next, we will show that the sequence $(\mathcal{U}_n^{IG})$ converges to $0$ if and only if~\eqref{iff-Un-IG} is satisfied. Indeed, by using inequality~\eqref{ln-ineq}, we have
			\begin{align*}
				\displaystyle (\alpha_1-1)\sum_{k=0}^{n} & \dfrac{a_k}{1-(1 - \alpha_1 )a_k}  +\sum_{k=0}^{n}\frac{\alpha_1 -1 + (\alpha_2 - \alpha_1)b_k}{\alpha_1 + (\alpha_2 - \alpha_1)b_k}  \le \ln(\mathcal{U}_n^{IG}) \\ & \le (\alpha_1-1)\sum_{k=0}^{n} a_k + \sum_{k=0}^{n} [(\alpha_1-1)+(\alpha_2 - \alpha_1)b_k] \\ &= (\alpha_1-1)\sum_{k=0}^{n}a_k + (\alpha_1-1)\sum_{k=0}^{n} (1 - b_k) + (\alpha_2-1) \sum_{k=0}^{n} b_k,
			\end{align*}
			On the other hand, thanks to Lemma~\ref{lem2}, we see that
			\begin{align*}
				\displaystyle \sum_{n \in \mathbb{N}} \dfrac{a_n}{1-(1 - \alpha_1)a_n} &\sim  \sum_{n \in \mathbb{N}} a_n, 
			\end{align*}
			and
			\begin{align*}
				\sum_{n \in \mathbb{N}} \dfrac{\alpha_1 - 1+(\alpha_2 - \alpha_1) b_n}{\alpha_1+(\alpha_2 - \alpha_1) b_n} &\sim  \sum_{n \in \mathbb{N}} [\alpha_1- 1+(\alpha_2 - \alpha_1)b_n] \\
				&\sim (\alpha_1-1) \sum_{n \in \mathbb{N}} (1 - b_n) + (\alpha_2-1) \sum_{n \in \mathbb{N}} b_n.
			\end{align*}
It allows us to conclude that~\eqref{iff-Un-IG} is the necessary and sufficient condition for the convergence of $\mathcal{U}^{IG}_{n}$ to $x^*$. We complete the proof.
		\end{proof}

\begin{theorem}\label{theo-L-IG}
Assume that  
\begin{align}\label{need-IG}
			a_k < \dfrac{1}{1 + \kappa_1} \ \text{ and } \ b_k < \dfrac{\kappa_1}{\kappa_1 + \alpha_2}, \ \text{ for all } k \in \mathbb{N}.
		\end{align}			
Then the optimal lower bound of~\eqref{xn-IG} is determined by
			\begin{align}\label{Ln-IG}
				\mathcal{L}_n^{IG} := \displaystyle \prod_{k=0}^n  [1 - (1+ \kappa_1)a_k][\kappa_1 - (\kappa_1 + \alpha_2)b_k],
			\end{align}
			for all $n \in \mathbb{N}$. Moreover, 
			\begin{enumerate}
				\item [(i)] If $\kappa_1 < 1$ then $\displaystyle \lim\limits_{n\to \infty} \mathcal{L}_n^{IG} = 0$ for every $(a_n), (b_n)$ satisfy~\eqref{need-IG}.
				\item [(ii)] If $\kappa_1 = 1$ then we have $\displaystyle \lim\limits_{n\to \infty} \mathcal{L}_n^{IG} = 0 \Longleftrightarrow \sum_{n \in \mathbb{N}} (a_n + b_n) = +\infty$. 
			\end{enumerate}
		\end{theorem}
		\begin{proof}
			Similarly, since $T_1 \in \text{Ne}_{\kappa_1, \alpha_1}(\Omega)$, $T_2 \in \text{Ne}_{\kappa_2, \alpha_2}(\Omega)$ and the parameter sequences $(a_n)$, $(b_n)$ satisfy the condition~\eqref{need-IG}, we can also have the following estimates
			\begin{align*}
				\| y_n - x^* \| &\ge (1 - a_n)\|x_n - x^*| - a_n\|T_1(x_n) - T_1(x^*)\| \\
				&\ge [1 - (1 + \kappa_1)a_n]\|x_n - x^*|, ~ n \in \mathbb{N},
			\end{align*}
			and
			\begin{align*} 
				\|x_{n+1}-x^*\| &\ge (1 - b_n)\|T_1(y_n) - T_1(x^*)\| - b_n\|T_2(y_n) - T_2(x^*)\| \\ &\ge [(1 - b_n)\kappa_1-b_n\alpha_2]\|y_n-x^*\|, \\
				&\ge [\kappa_1 - (\kappa_1 + \alpha_2)b_n][1 - (1 + \kappa_1)a_n]\|x_n - x^*|, ~ n \in \mathbb{N}.
			\end{align*}
			It leads to $\|x_{n+1}-x^*\| \ge \mathcal{L}_n^{IG}\|x_1-x^*\|$, for all $n \in \mathbb{N}$,
			where $\mathcal{L}_n^{IG}$ is defined by $\eqref{Ln-IG}$. 	On the other hand, we can easily check that if we choose two mappings $T_1,T_2$ defined by
			\begin{align*}
				T_1(x)= \kappa_1 x, \ T_2(x)=-\alpha_2 x, \quad x \in B\left(0, \dfrac{1}{2}\right),
			\end{align*}
			then we can observe that $x^*=0$ and 
			$$\|x_{n+1}-x^*\|=\mathcal{L}_n^{IG}\|x_0-x^*\|,$$
			for every $n \in \mathbb{N}$. Therefore, $\mathcal{L}_n^{IG}$ is the optimal lower bound of the iteration~\eqref{xn-IG}. \\
			
			In the first case $\kappa_1<1$, for all $k \in \mathbb{N}$, we have
			\begin{align*}
				[\kappa_1 - (\kappa_1 + \alpha_2)b_k][1 - (1 + \kappa_1)a_k] \le \kappa_1,~ \text{for all} ~k \in \mathbb{N}.
			\end{align*}
			It leads to $\mathcal{L}_n^{IG}\le \kappa_1^n$. Thus, we deduce that $\lim{\mathcal{L}_n^{IG}}=0$. \\
			
			For the second case $\kappa_1=1$, for every $n \in \mathbb N^*$, $\ln \mathcal{L}_n^{IG}$ can be rewritten as below
			\begin{align*}
				\ln{\mathcal{L}_n^{IG}} = \sum_{k=0}^{n} \ln[1 - (1 + \alpha_2)b_n] + \sum_{k=0}^{n} \ln(1 - 2a_n), ~ n \in \mathbb{N}.
			\end{align*}
By the inequality~\eqref{ln-ineq}, we have two following estimates
			\begin{align*}
				&\ln{\mathcal{L}_n^{IG}} \le -(1+\alpha_2)\sum_{k=0}^{n}{b_k}-2\sum_{k=0}^{n}{a_k},\\
				&\ln{\mathcal{L}_n^{IG}} \ge -(1+\alpha_2)\sum_{k=0}^{n} \dfrac{b_k}{1-(1+\alpha_2)b_k} - 2\sum_{k=0}^{n} \dfrac{a_k}{1-2a_k}.
			\end{align*}
Moreover, Lemma~\ref{lem2} gives us
			\begin{align*} 
				\sum_{k \in \mathbb{N}} \dfrac{b_k}{1-(1+\alpha_2)b_k} + \sum_{k \in \mathbb{N}} \dfrac{a_k}{1-2a_k}  \sim \sum_{k \in \mathbb{N}}a_k+ \sum_{k \in \mathbb{N}}b_k.
			\end{align*}
It allows us to conclude the second statement of this theorem. 
		\end{proof}
		\begin{corollary}\label{cor-IG}
			Under the assumptions of Theorem~\ref{theo-L-IG}, the sequence $(x_n^{IG})$ converges to $x^*$ if the following assumption holds true
			$$\displaystyle (1-\alpha_1)\sum_{k \in \mathbb{N}}(a_k+1-b_k)+(1-\alpha_2)\sum_{k \in \mathbb{N}}b_k=+\infty.$$
		\end{corollary}

\subsection{Generalized two-step iteration (G)}

Let us consider six constants $\alpha_1,\alpha_2,\beta_1,\beta_2, \kappa_1, \kappa_2 \in [0,1]$ such that $0 < \alpha_1 + \alpha_2 + \beta_1 + \beta_2 < 4$, $\kappa_1 \le \beta_1$, $\kappa_2 \le \beta_2$ and two parameter sequences $(a_n),(b_n) \subset [0, 1]$. For every four mappings $S_1 \in \text{Ne}_{\kappa_1, \beta_1}(\Omega)$, $S_2 \in \text{Ne}_{\kappa_2, \beta_2}(\Omega)$, $T_1 \in \text{Ne}_{\alpha_1}(\Omega)$, $T_2 \in \text{Ne}_{\alpha_2}(\Omega)$ with the common fixed point $x^* \in \mathcal{F}(S_1, S_2, T_1, T_2)$, the iterative sequence $(x_n)$ is defined by
\begin{align}\tag{G}\label{xn-G}
\begin{cases}x_0 \in \Omega,\\ y_{n} = (1-a_n)S_1(x_n) + a_nT_1(x_n), \\ x_{n+1} = (1-b_n)S_2(y_n) + b_nT_2(y_n), \ n \in \mathbb{N},\end{cases}	
\end{align} 
which is called the generalized iteration (G).

\begin{theorem}\label{theo-U-G}
The optimal upper bound of~\eqref{xn-G} is determined by
\begin{align}\label{Un-G}
\mathcal{U}_n^{G}=\prod_{k=0}^n {[\beta_1+(\alpha_1 - \beta_1)a_k][\beta_2+(\alpha_2-\beta_2)b_k]},
\end{align}
for all $n \in \mathbb{N}$. Moreover, the sequence $(\mathcal{U}_n^G)$ converges to 0 if and only if
\begin{align}\label{iff-Un-G}
(1-\alpha_1)\sum_{k \in \mathbb{N}} a_k & + (1-\alpha_2)\sum_{k \in \mathbb{N}} b_k \notag \\
& + (1-\beta_1) \sum_{k \in \mathbb{N}} (1 - a_k) + (1-\beta_2)\sum_{k \in \mathbb{N}} (1- b_k) = +\infty. 
\end{align}
\end{theorem}
\begin{proof}
By the similar technique used for the~\eqref{xn-IG} iteration, we also obtain that
\begin{align*}
\|y_n-x^*\| \le [(\beta_1 + (\alpha_1 - \beta_1)a_n]\|x_n - x^*\|,
\end{align*}
and
\begin{align*}
\|x_{n+1}-x^*\| \le [\beta_2 + (\alpha_2 - \beta_2)b_n][(\alpha_1 + (\beta_1 - \alpha_1)a_n]\|x_n - x^*\|,
\end{align*}
for all $n \in \mathbb{N}$. It yields that $\|x_{n+1}-x^*\| \le \mathcal{U}_n^{G}\|x_0-x^*\|$ for all $n \in \mathbb{N}$, where $\mathcal{U}_n^{G}$ is defined by~\eqref{Un-G}. Moreover, if we choose $\Omega=B\left(0,\dfrac{1}{2}\right)$ and four mappings $S_1,S_2,T_1,T_2$ defined by
\begin{align*}
S_1(x)=\beta_1 x, \ S_2(x)=\beta_2x, \ T_1(x)=\alpha_1x, \ T_2(x)=\alpha_2x, \quad x \in \Omega,
\end{align*}
then we have $x^*=0$ and $\|x_{n+1}-x^*\|=\mathcal{U}_n^G\|x_0-x^*\|$, for every $n \in \mathbb{N}$. Therefore, $\mathcal{U}_n^G$ is the optimal upper bound of (G). \\

Next, we prove that the sequence $(\mathcal{U}_n^G)$ converges to $0$ if and only if~\eqref{iff-Un-G} is satisfied. By the inequality introduced at~\eqref{ln-ineq}, we have
\begin{align*}
\ln \mathcal{U}_n^G \ge \sum_{k=0}^{n} \dfrac{\beta_1 - 1 + (\alpha_1 - \beta_1)a_k}{\beta_1 + (\alpha_1 - \beta_1)a_k} + \sum_{k=0}^{n} \dfrac{\beta_2 - 1 + (\alpha_2 - \beta_2)b_k}{\beta_2 + (\alpha_2 - \beta_2)b_k}, 
\end{align*}
and
\begin{align*}
\ln \mathcal{U}_n^G &\le \sum_{k=0}^{n} [\beta_1 - 1 + (\alpha_1 - \beta_1)a_k] + \sum_{k=0}^{n} [\beta_2 - 1 + (\alpha_2 - \beta_2)b_k] \\
&= (\alpha_1 - 1)\sum_{k=0}^{n} a_k + (\alpha_2 - 1)\sum_{k=0}^{n} b_k \\
& \qquad \qquad + (\beta_1 - 1) \sum_{k=0}^{n} (1 - a_k) + (\beta_2 - 1)\sum_{k=0}^{n} (1 - b_k).
\end{align*}
On the other hand, by Lemma~\ref{lem2}, we have
\begin{align*}
\sum_{k \in \mathbb{N}} \dfrac{\beta_1 - 1 + (\alpha_1 - \beta_1)a_k}{\beta_1 + (\alpha_1 - \beta_1)a_k} &\sim  \sum_{k \in \mathbb{N}} [\beta_1 - 1 + (\alpha_1 - \beta_1)a_k] \\
&\sim (\alpha_1 - 1) \sum_{k \in \mathbb{N}} a_k + (\beta_1 - 1) \sum_{k \in \mathbb{N}} (1- a_k),
\end{align*}
and 
\begin{align*}
\sum_{k \in \mathbb{N}} \dfrac{\beta_2 - 1 + (\alpha_2 - \beta_2)b_k}{\beta_2 + (\alpha_2 - \beta_2)b_k} &\sim  \sum_{k \in \mathbb{N}} [\beta_2 - 1 + (\alpha_2 - \beta_2)b_k] \\
&\sim  (\alpha_2 - 1) \sum_{k \in \mathbb{N}} b_k + (\beta_2 - 1) \sum_{k \in \mathbb{N}} (1- b_k).
\end{align*}
From the above statements, we conclude that~\eqref{iff-Un-G} is the necessary and sufficient condition for the convergence of $(\mathcal{U}_n^G)$. 
\end{proof}

\begin{theorem}\label{theo-L-G}
Suppose further that 
\begin{align}\label{need-G}
a_n \le \dfrac{\kappa_1}{\kappa_1 + \alpha_1}, \quad b_n \le \dfrac{\kappa_2}{\kappa_2 + \alpha_2}, \quad n \in \mathbb{N}.
\end{align}
Then the optimal lower bound of~\eqref{xn-G} is determined by 
\begin{align}\label{Ln-G}
\mathcal{L}_n^{G}=\prod_{k=0}^n {[\kappa_1-(\kappa_1+\alpha_1)a_k][\kappa_2-(\kappa_2+\alpha_2)b_k]},
\end{align}
for all $n \in \mathbb{N}$. Moreover,
\begin{enumerate}
\item [(i)] If $\kappa_1 < 1$ or $\kappa_2 < 1$ then $\displaystyle \lim\limits_{n\to \infty} \mathcal{L}_n^G = 0$ for every $(a_n), (b_n)$ satisfy~\eqref{need-G}.
\item [(ii)] If $\kappa_1 = \kappa_2 = 1$ then we have $\displaystyle \lim\limits_{n\to \infty} \mathcal{L}_n^G = 0 \Longleftrightarrow \sum_{n \in \mathbb{N}} (a_n + b_n) = +\infty$.
\end{enumerate}
\end{theorem}
\begin{proof}
By the similar technique presented in Theorem~\ref{theo-L-IG}, we can obtain the estimate for the sequence $(x_n^G)$ as below
\begin{align*} 
\|x_{n+1}-x^*\| \ge [\kappa_1-(\kappa_1+\alpha_1)a_n][\kappa_2-(\kappa_2+\alpha_2)b_n]\|x_n-x^*\|,
\end{align*}
for all $n \in \mathbb{N}$. It leads to $\|x_{n+1}-x^*\| \ge \mathcal{L}_n^{G}\|x_0-x^*\|$, for all $n \in \mathbb{N}$,
where $\mathcal{L}_n^{G}$ is defined by $\eqref{Ln-G}$. By choosing four mappings $S_1,S_2,T_1,T_2$ defined by
\begin{align*}
S_1(x)=\kappa_1 x,  ~ S_2(x)=\kappa_2 x, ~ T_1(x)= \alpha_1 x, ~ T_2(x)=\alpha_2 x, ~ x \in B\left(0, \dfrac{1}{2}\right),
\end{align*}
we can observe that $x^*=0$ and $\|x_{n+1}-x^*\|=\mathcal{L}_n^{G}\|x_0-x^*\|$ for every $n \in \mathbb{N}$. Therefore, $\mathcal{L}_n^G$ is the optimal lower bound of the iteration~\eqref{xn-G}.\\

Let us consider the first case $\kappa_1<1$ or $\kappa_2<1$. For all $k \in \mathbb{N}$, we have
\begin{align*}
[\kappa_1-(\kappa_1+\alpha_1)a_k][\kappa_2-(\kappa_2+\alpha_2)b_k]\le \kappa_1  \kappa_2 \le \min\{\kappa_1,\kappa_2\},~ \text{for all} ~k \in \mathbb{N}.
\end{align*}
It leads to $\mathcal{L}_n^G\le (\min\{\kappa_1,\kappa_2\})^n$. Combining with the fact that $\min\{\kappa_1,\kappa_2\}<1$, we deduce that $\lim{\mathcal{L}_n^G}=0$.\\

Otherwise, if $\kappa_1=\kappa_2=1$, $\ln \mathcal{L}_n^G$ can be rewritten as below
\begin{align*}
\ln{\mathcal{L}_n^G} = \sum_{k=0}^{n}{\ln(1-(1+\alpha_1)a_k)}+\sum_{k=0}^{n}{\ln(1-(1+\alpha_2)b_k)},
\end{align*}
for every $n \in \mathbb{N}$. By the inequality~\eqref{ln-ineq}, we have two following estimates
\begin{align*}
&\ln{\mathcal{L}_n^G} \le -(1+\alpha_1)\sum_{k=0}^{n}{a_k}-(1+\alpha_2)\sum_{k=0}^{n}{b_k},\\
&\ln{\mathcal{L}_n^G} \ge -(1+\alpha_1)\sum_{k=0}^{n}{\dfrac{a_k}{1-(1+\alpha_1)a_k}}-(1+\alpha_2)\sum_{k=0}^{n}{\dfrac{b_k}{1-(1+\alpha_2)b_k}}.
\end{align*}
Applying Lemma~\ref{lem2} again, there holds
\begin{align*} 
\sum_{k \in \mathbb{N}} {\dfrac{a_k}{1-(1+\alpha_1)a_k}} + \sum_{k \in \mathbb{N}} {\dfrac{b_k}{1-(1+\alpha_2)b_k}} \sim \sum_{k \in \mathbb{N}}a_k+ \sum_{k \in \mathbb{N}}b_k,
\end{align*}
which allows us to conclude the last statement of this theorem. 
\end{proof}

\begin{corollary}\label{cor-G}
Under the assumptions of Theorem \ref{theo-L-G}, the sequence $(x_n^G)$ converges to $x^*$ if one of two following statements holds true:
\begin{itemize}
\item [(i)] $\min\{\kappa_1,\kappa_2\} <1$.
\item [(ii)] $\kappa_1=\kappa_2=1,$ and $\displaystyle \sum_{n \in\mathbb{N}}{(a_n+b_n)}=+\infty.$
\end{itemize}
Moreover, if $\alpha_1,\alpha_2 \in (0,1)$ and the conditions in $(ii)$ are satisfied, we have the following statement
\begin{align} \label{dl-1.7}
\lim_{n \to \infty}{x_n}=x^* \Longleftrightarrow \sum\limits_{i\in \mathbb N}{(a_i+b_i)}=+\infty.
\end{align}
\end{corollary}
\begin{proof}
Assuming that condition in $(i)$ holds, it allows us to arrive 
\begin{align*}
[\kappa_1-(\kappa_1+\alpha_1)a_n][\kappa_2-(\kappa_2+\alpha_2)b_n] \le  \kappa_1  \kappa_2 \le \min\{\kappa_1,\kappa_2\}.
\end{align*}
As a consequence, $\mathcal{L}_n^G \le (\min\{\kappa_1,\kappa_2\})^n$, for all $n \in \mathbb{N}$. Since $\min\{\kappa_1,\kappa_2\} < 1$, one gets that $\lim {\mathcal{L}_n^G}=0$. For this reason, we conclude that the sequence $(x_n^G)$ converges to $x^*$.\\

Finally, if the group of assumptions $(ii)$ is satisfied, according to the proof of Theorem~\ref{theo-L-G}, we can easily come to the same conclusion. Furthermore, if $\alpha_1,\alpha_2 \in (0,1)$ and the conditions in $(ii)$ are satisfied, by combining Theorem~\ref{theo-U-G} and Theorem~\ref{theo-L-G}, we obtain that
\begin{align*}
\sum_{n \in\mathbb{N}}{(a_n+b_n)} = +\infty \Longleftrightarrow \lim{\mathcal{U}_n^G}=\lim{\mathcal{L}_n^G}=0,
\end{align*}
which confirms that the statement~\eqref{dl-1.7} is proved.
\end{proof}

\section{Comparison results}\label{Sec4}

In this section, we apply the optimal error bounds to compare the convergence rate of two iterative processes. Let us consider two mappings $\mathcal{P}, \mathcal{Q}: \Omega \to \Omega$ such that $\mathcal{P}(x^*)=\mathcal{Q}(x^*)=x^*$. With an initial value $x_0 = u_0 \in \Omega$, we define two iterative sequences
\begin{align}\notag
x_{n+1} = \mathcal{P}(x_n), \ \mbox{ and } \ 
u_{n+1} = \mathcal{Q}(u_n), \quad n \in \mathbb{N}.
\end{align}
We further assume that both sequences $(x_n)_{n \in \mathbb{N}}$ and  $(u_n)_{n \in \mathbb{N}}$ converge to $x^*$ in $\mathbb{X}$ as $n$ tends to infinity. It is well-known that the convergence rate of the iterative process associated to $\mathcal{P}$ can be estimated by the following ratio 
$$I_n(\mathcal{P}) := \frac{\|x_n-x^*\|}{\|x_0-x^*\|}, \quad n \in \mathbb{N}.$$ 
In order to compare the convergence speed of two iterative processes, we simply compare two convergence rates. More precisely, we define a new ratio
\begin{align}\label{ratio-R}
\mathcal{R}(x_n,u_n,x^*) := \frac{I_n(\mathcal{P})}{I_n(\mathcal{Q})} = \frac{\|x_n-x^*\|}{\|u_n-x^*\|}, \quad n \in \mathbb{N}.
\end{align}
For the validation of $\mathcal{R}$ in some singular cases, we set $\mathcal{R}(x_n,u_n,x^*) = 0$ if $x_n = u_n = x^*$ and $\mathcal{R}(x_n,u_n,x^*) = 1$ if $u_n = x^*$ but $x_n \neq x^*$.

\begin{definition}[Comparison of convergence rates]
\label{def-conv-rate}
We say that the iterative process associated to $\mathcal{P}$ converges faster the iterative process associated to $\mathcal{Q}$ if and only if
\begin{align}\notag 
\lim\limits_{n \to \infty} \mathcal{R}(x_n,u_n,x^*) = 0.
\end{align}
For simplicity, we say that $x_n$ converges to $x^*$ faster than $u_n$.
\end{definition}

\subsection{Comparison between (IG) and (I), (IM)}

By comparing the optimal upper bound of the iteration~\eqref{xn-IG} with the optimal lower bound of the iterations~\eqref{xn-IM} and~\eqref{xn-I}, we can propose a sufficient condition on the parameter sequences $(a_n)$ and $(b_n)$ so that the iteration~\eqref{xn-IG} converges faster than the iterations~\eqref{xn-I} and~\eqref{xn-IM}. 
\begin{theorem}\label{theo-IG-and-I}
Given two constants $\alpha_1,\alpha_2 \in [0,1]$ such that $0<\alpha_1+\alpha_2<2$ and two parameter sequence $(a_n)$ and $(b_n)$ in $[0,1]$. Suppose that there exists a constant $\delta>0$ such that
\begin{align}\label{cond:I-IG-1}
	1-b_k-\alpha_2b_k(1-a_k+\alpha_1a_k)\ge\delta, \ \text{ for all} \ k \in \mathbb{N},
\end{align}
and
\begin{align}\label{cond:I-IG-2}
	\displaystyle\sum_{k \in \mathbb{N}}{\left[(1-\alpha_1)(a_k+1)+\left(\alpha_1-\alpha_2-\displaystyle\dfrac{1+\alpha_2}{\delta}\right)b_k\right]}=+\infty.
\end{align}
Then the iteration~\eqref{xn-IG} converges faster than the iteration~\eqref{xn-I} if one of the following two groups of hypotheses holds true
\begin{itemize}
\item[i)] $\alpha_1 \in [0,1]$, $\alpha_2 \in [0,1)$ and $\displaystyle\sum_{k \in \mathbb{N}}{b_k}=+\infty$.
\item[ii)] $\alpha_1 \in [0,1)$, $\alpha_2=1$, $\displaystyle\sum_{k \in \mathbb{N}}(a_k-b_k+1)=+\infty$ and $\displaystyle\sum_{k \in \mathbb{N}}a_kb_k=+\infty$.
\end{itemize}
\end{theorem}
\begin{proof}
Let us assume that the iterations~\eqref{xn-IG} and~\eqref{xn-I} determine two sequences $x_n^{IG}$ and $x_n^{I}$ respectively. Since i) or ii) hold, Corollary~\ref{cor-IG} and \cite[Theorem 4.1]{paper2} ensure that $x_n^{IG}$ and $x_n^{I}$ converge to $x^*$. Combining with~\eqref{cond:I-IG-1}, it ensures that the sequences $\left(\mathcal{U}^{IG}_n\right)$ and $\left(\mathcal{L}^I_n\right)$ are the optimal upper bound of the iteration~\eqref{xn-IG} and the optimal lower bound of the iterations~\eqref{xn-I}, respectively. Then, by the definition of $\mathcal{R}$ in~\eqref{ratio-R}, there holds  
\begin{align}\label{est-IG-I}
\mathcal{R}(x_n^{IG},x_n^{I},x^*) =	\displaystyle\dfrac{\|x_n^{IG}-x^*\|}{\|x_n^{I}-x^*\|} \le \beta_n :=\displaystyle\dfrac{\mathcal{U}_n^{IG}}{\mathcal{L}_n^{I}}, \quad \mbox{ for all } n \in \mathbb{N}.
\end{align}
By using the elementary inequality~\eqref{ln-ineq} in Lemma~\ref{lem1}, for each $n \in \mathbb{N}$, one gets the following estimate
\begin{align*}
	\ln(\beta_n) &=\ln\left(\displaystyle\prod_{k=0}^n {\dfrac{[\alpha_1(1-b_k)+\alpha_2b_k][1-(1-\alpha_1)a_k]}{1-b_k-\alpha_2b_k[1-(1-\alpha_1)a_k]}}\right) \\
	&=\displaystyle\sum_{k=1}^{n}\left(\ln[\alpha_1(1-b_k)+\alpha_2b_k]-\ln\left(1-b_k-\alpha_2b_k[1-(1-\alpha_1)a_k]\right)\right) \\
	& \qquad \qquad +\displaystyle\sum_{k=1}^{n}\ln\left[1-(1-\alpha_1)a_k\right]\\
	&\le \displaystyle\sum_{k=1}^{n}[\alpha_1(1-b_k)+\alpha_2b_k-1]-\displaystyle\sum_{k=1}^{n}{\displaystyle\dfrac{-b_k-\alpha_2b_k[1-(1-\alpha_1)a_k]}{1-b_k-\alpha_2b_k[1-(1-\alpha_1)a_k]}} \\
	& \qquad \qquad -\displaystyle\sum_{k=0}^{n}{(1-\alpha_1)a_k} \\
	&=\displaystyle\sum_{k=1}^{n}[(\alpha_2-\alpha_1)b_k-(1-\alpha_1)]+\displaystyle\sum_{k=1}^{n}{\displaystyle\dfrac{b_k+\alpha_2b_k[1-(1-\alpha_1)a_k]}{1-b_k-\alpha_2b_k[1-(1-\alpha_1)a_k]}} \\ 
	& \qquad \qquad -\displaystyle\sum_{k=0}^{n}{(1-\alpha_1)a_k} \\
	&\le - \displaystyle\sum_{k \in \mathbb{N}}{\left[(1-\alpha_1)(a_k+1)+\left(\alpha_1-\alpha_2-\displaystyle\dfrac{1+\alpha_2}{\delta}\right)b_k\right]}. 
\end{align*}
Thanks to~\eqref{cond:I-IG-2}, we have $\lim\limits_{n\to\infty}{\beta_n}=0$. Therefore, by~\eqref{est-IG-I} we may conclude that
\begin{align*}
\lim\limits_{n\to\infty} \mathcal{R}(x_n^{IG},x_n^{I},x^*) = 0.
\end{align*}
That means the iteration~\eqref{xn-IG} converges faster than the iteration~\eqref{xn-I} in the sense of Definition~\ref{def-conv-rate}.
\end{proof}

\begin{theorem}\label{theo-IG-and-IM}
Given two constants $\alpha_1,\alpha_2 \in [0,1]$ such that $0<\alpha_1+\alpha_2<2$ and two parameter sequences $(a_n)$ and $(b_n)$ in $[0,1]$. Assume that there exist two positive constant $\varepsilon, \tau$ such that
\begin{align}\label{cond:IG-IM-1}
	1-a_k-\alpha_1a_k \ge \varepsilon,\ 1-b_k-\alpha_2b_k \ge \tau, \quad \text{for all } k \in \mathbb{N},
\end{align}
and 
\begin{align}\label{cond:IG-IM-2}
	\displaystyle\sum_{k \in \mathbb{N}}{\left[(1-\alpha_1)a_k+\dfrac{\alpha_1-\alpha_2}{\varepsilon\tau}b_k+1-\displaystyle\dfrac{\alpha_1}{\varepsilon\tau}\right]}=+\infty.
\end{align}
Then the iteration~\eqref{xn-IG} converges faster than the iteration~\eqref{xn-IM} if one of the following two groups of hypotheses holds 
\begin{itemize}
\item[i)] $\alpha_1 \in [0,1]$, $\alpha_2 \in [0,1)$ and $\displaystyle\sum_{k \in \mathbb{N}}{b_k}=+\infty$.
\item[ii)] $\alpha_1 \in [0,1)$, $\alpha_2=1$, $\displaystyle\sum_{k \in \mathbb{N}}a_k=+\infty$ and $\displaystyle\sum_{k \in \mathbb{N}}(a_k-b_k+1)=+\infty$.
\end{itemize}
\end{theorem}
\begin{proof}
Similar to the previous proof, the iterations~\eqref{xn-IG} and~\eqref{xn-IM} determine two sequences $x_n^{IG}$ and $x_n^{IM}$ respectively, which converge to $x^*$ provided $i)$ or $ii)$. Moreover,~\eqref{xn-IG} admits the optimal upper bound $\left(\mathcal{U}^{IG}_n\right)$ and~\eqref{xn-IM} admits the optimal lower bound $\left(\mathcal{L}^{IM}_n\right)$ under assumptions~\eqref{cond:IG-IM-1} and~\eqref{cond:IG-IM-2}. By setting $\gamma_n=\displaystyle\dfrac{\mathcal{U}_n^{IG}}{\mathcal{L}_n^{IM}}$ for each  $n \in \mathbb{N}$, one has  
\begin{align}\label{est-IG-IM-2}
\mathcal{R}(x_n^{IG},x_n^{IM},x^*) \le \gamma_n,\ \mbox{ for all } n \in \mathbb{N}.
\end{align}
We now apply~\eqref{cond:IG-IM-1} and the following elementary inequality 
$$\ln(x)-\ln(y) \le \dfrac{x-y}{y}, \ \mbox{ for every } x,y >0,$$ 
we get the following estimate
\begin{align*}
	\ln(\gamma_n) &=\ln\left(\displaystyle\prod_{k=0}^n {\dfrac{[\alpha_1(1-b_k)+\alpha_2b_k][1-(1-\alpha_1)a_k]}{[1-(1+\alpha_2)b_k][1-(1+\alpha_1)a_k]}}\right) \\
	&=\displaystyle\sum_{k=0}^{n}{\ln[\alpha_1(1-b_k)+\alpha_2b_k]-\ln\left([1-(1+\alpha_2)b_k][1-(1+\alpha_1)a_k]\right)} \\
	& \qquad \qquad +\displaystyle\sum_{k=0}^{n}{\ln[1-(1-\alpha_1)a_k]} \\
	&\le\displaystyle\sum_{k=0}^{n}{\dfrac{\alpha_1(1-b_k)+\alpha_2b_k-[1-(1+\alpha_2)b_k][1-(1+\alpha_1)a_k]}{[1-(1+\alpha_2)b_k][1-(1+\alpha_1)a_k]}} \\
	& \qquad \qquad -\displaystyle\sum_{k=0}^{n}{(1-\alpha_1)a_k} \\
	&\le -\displaystyle\sum_{k \in \mathbb{N}}{\left[(1-\alpha_1)a_k+\dfrac{\alpha_1-\alpha_2}{\varepsilon\tau}b_k+1-\displaystyle\dfrac{\alpha_1}{\varepsilon\tau}\right]},
\end{align*}
for $n \in \mathbb{N}$. Combining~\eqref{est-IG-IM-2} with assumption~\eqref{cond:IG-IM-2}, it implies to 
\begin{align*}
\lim\limits_{n\to\infty} \mathcal{R}(x_n^{IG},x_n^{I},x^*) =  \lim\limits_{n\to\infty}{\gamma_n} = 0.
\end{align*}
We conclude that the iteration~\eqref{xn-IG} converges faster than~\eqref{xn-IM}.
\end{proof}

\subsection{Comparison between (G) and (I), (IM)}

Now, we will compare the convergence speed of three iterative processes (G), (I), and (IM). To avoid mistakes, we denote by $(u_n)$, $(x_n)$ and $(w_n)$ the iterative sequences created by processes (G), (I) and (IM) respectively, this means that 
\begin{align}\label{xn-G-2.2}
\begin{cases}
	u_0 \in \Omega, \\
	v_n=(1-a_n)S_1(u_n)+a_nT_1(u_n),\\
	u_{n+1}=(1-b_n)S_2(v_n)+b_nT_2(v_n), \end{cases}
\end{align}
\begin{align}\label{xn-I-2.2}
\begin{cases}
	x_0 \in \Omega, \\
	y_n=(1-a_n)x_n+a_nT_1(x_n),\\
	x_{n+1}=(1-b_n)x_n+b_nT_2(y_n), \end{cases}
\end{align}
and
\begin{align}\label{xn-IM-2.2}
\begin{cases}
	w_0 \in \Omega, \\
	z_n=(1-a_n)w_n+a_nT_1(w_n),\\
	w_{n+1}=(1-b_n)z_n+b_nT_2(z_n), \end{cases}
\end{align}
for $n \in \mathbb{N}$. Here, we assume that they have the same initial value $x_0 = w_0 = u_0$.  Let us state the following results.
\begin{theorem}
Let $(u_n)$, $(x_n)$ and $(w_n)$ be sequences defined as in~\eqref{xn-G-2.2},~\eqref{xn-I-2.2} and~\eqref{xn-IM-2.2}  respectively. Suppose that two parameter sequences $(a_n)$, $(b_n)$ belong to the following ranges
\begin{align}\label{cond-2.2}
0 \le a_n < \dfrac{1 - \beta_1}{1 - \beta_1 + 2\alpha_1}, \quad \  0 \le b_n < \dfrac{1 - \beta_2}{1 - \beta_2 + 2\alpha_2}.
\end{align}
\begin{itemize}
\item[i)] If $\displaystyle\sum_{n \in \mathbb{N}} b^*_n = +\infty$ then $(u_n)$ converges faster than $(x_n)$. 
\item[ii)] If $\displaystyle\sum_{n \in \mathbb{N}} (a^*_n + b^*_n) = +\infty$ then the  $(u_n)$ converges faster than $(w_n)$. 
\end{itemize}
Here, two shifted sequences $a^*_n$ and $b_n^*$ are defined by
\begin{align}\notag 
a^*_n = \dfrac{1 - \beta_1}{1 - \beta_1 + 2\alpha_1} - a_n, \quad \ b^*_n = \dfrac{1 - \beta_2}{1 - \beta_2 + 2\alpha_2} - b_n, \quad n \in \mathbb{N}.
\end{align}
\end{theorem}
\begin{proof}
It is worth mentioning that $u_n$ and $x_n$ converge to $x^*$ under assumptions in this theorem. Using the formulas of optimal upper bound of (G) and optimal lower bound of (I), we have the following estimate
\begin{align*}
\mathcal{R}(u_n,x_n,x^*) & \le \prod_{k=0}^n \dfrac{\left[\beta_2-(\beta_2-\alpha_2)b_k\right]\left[\beta_1-(\beta_1-\alpha_1)a_k\right]}{1 - b_k - \alpha_2b_k[1 - (1-\alpha_1)a_k]} \\
& \le \prod_{k=0}^n \dfrac{\beta_2-(\beta_2-\alpha_2)b_k}{1 - (1+\alpha_2)b_k}.
\end{align*}
The conditions in~\eqref{cond-2.2} guarantee that \begin{align*}
	0< 1-(1+\alpha_2)b_n \text{ and } 0< \beta_2 - (\beta_2 - \alpha_2) b_n, \text{ for all } n \in \mathbb{N}.
\end{align*} 
Therefore, with the similar techniques as in the proof presented above, we obtain the first conclusion in $i)$. Similarly, to compare the convergence speed between two iterative processes (G) and (IM), we use the following estimate
\begin{align*}
\mathcal{R}(u_n,w_n,x^*) &\leq \prod_{k=0}^n \dfrac{\left[\beta_2-(\beta_2-\alpha_2)b_k\right]\left[\beta_1-(\beta_1-\alpha_1)a_k\right]}{[1 - (1+\alpha_2)b_k][1 - (\alpha_1 + 1)a_k]}.
\end{align*}
We also have the last statement in $ii)$.
\end{proof}

\subsection{Comparison between (G) and (IG)}

To compare the convergence rate of~\eqref{xn-IG} and~\eqref{xn-G}, we assume that the iteration processes~\eqref{xn-IG} and~\eqref{xn-G} determine sequences $(x_n^{IG})$ and $(x_n^G)$ respectively with the same initial guess $x_0^{IG}=x_0^G$.
\begin{theorem}
Given five constants $\alpha_1,\kappa_1,\alpha_2,\beta_1,\beta_2\in (0,1]$ such as $\beta_2\leq \kappa_1\leq \alpha_1$ and two parameter consequences $(a_n),(b_n)$ contained in $[0,1]$. Assume that 
\begin{align}\label{cond-abn}
	a_n\leq \dfrac{1-\beta_1}{1-\beta_1+2\alpha_1} \quad \text{and} \quad b_n\leq \dfrac{\kappa_1-\beta_2}{\kappa_1-\beta_2+2\alpha_2}.
\end{align}
Let $(x_n^{IG})$ and $(x_n^G)$ be sequences determined by (IG) and (G) respectively, with $T_1\in \mathcal{N}_{\kappa_1,\alpha_1}(\Omega)$, $T_2\in \mathcal{N}_{\alpha_2}(\Omega)$, $S_1\in \mathcal{N}_{\beta_1}(\Omega)$, $S_2\in\mathcal{N}_{\beta_2}(\Omega)$ and $x^*\in \mathcal{F}(T_1,T_2,S_1,S_2)$. 
Then the sequence $(x_n^G)$ converges to $x^*$ faster than $(x_n^{IG})$, provided 
\begin{align}\label{ser}
	\displaystyle\sum_{n \in \mathbb{N}}\left(\dfrac{1-\beta_1}{1-\beta_1+2\alpha_1}-a_n\right) + \displaystyle\sum_{n \in \mathbb{N}}\left(\dfrac{\kappa_1-\beta_2}{\kappa_1-\beta_2+2\alpha_2}-b_n\right)=+\infty.
\end{align} 
\end{theorem}
\begin{proof}
By using the formulas of optimal upper bound of~\eqref{xn-G} and optimal lower bound of~\eqref{xn-IG}, we have the estimate below: \begin{align}\label{est-2.3}
\mathcal{R}(x_n^G,x_n^{IG},x^*)	\le \prod_{k=0}^n \dfrac{\left[\beta_1-(\beta_1-\alpha_1)a_k\right]\left[\beta_2-(\beta_2-\alpha_2)b_k\right]}{\left[1-(1+\alpha_1)a_k\right]\left[\kappa_1-(\kappa_1+\alpha_2)b_k\right]}=\prod_{k=0}^n a^*_k b^*_k,
\end{align}
where 
\begin{align}\notag 
	a^*_n=\frac{\beta_1-(\beta_1-\alpha_1)a_n}{1-(1+\alpha_1)a_n}, \quad  b^*_n=\dfrac{\beta_1-(\beta_1-\alpha_1)a_n}{\kappa_1-(\kappa_1+\alpha_2)b_n},
\end{align}
for all $n \in \mathbb{N}$. The conditions in~\eqref{cond-abn} guarantee that \begin{align*}
	1-(1+\alpha_1)a_n\geq \beta_1-(\beta_1-\alpha_1)a_n>0,
\end{align*} 
and
\begin{align*}
\kappa_1-(\kappa_1+\alpha_2)b_n\geq \beta_2-(\beta_2-\alpha_2)b_n>0.
\end{align*} 
This ensures two sequences $(a^*_n)$ and $(b^*_n)$ are in $(0,1]$.
On the other hand, \begin{align}\label{con-2.3}
	\ln \prod_{k=0}^n  a^*_kb^*_k=\displaystyle\sum_{k=0}^{n}\ln (a^*_k b^*_k)=\displaystyle\sum_{k=0}^{n}\ln a^*_k+\displaystyle\sum_{k=0}^{n}\ln b^*_k.
\end{align}
Since $(a^*_n),(b^*_n)\subset (0,1]$, we have $\ln a^*_k\leq 0$, $\ln b^*_k\leq 0$, for all $k \in \mathbb{N}$. Therefore, from~\eqref{con-2.3}, we deduce that \begin{align*}
	\lim\limits_{n\to \infty}\ln \prod_{k=0}^{n} a^*_kb^*_k=-\infty &\iff \sum_{k \in \mathbb{N}}\ln a^*_k=-\infty \text{ or } \sum_{k \in \mathbb{N}}\ln b^*_k=-\infty,
\end{align*}
which means \begin{align}\label{cond-2.3-1}
	\lim\limits_{n\to \infty}\prod_{k=0}^n a^*_kb^*_k=0\iff \lim\limits_{n\to \infty}\prod_{k=0}^n a^*_k=0 \text{ or } \lim\limits_{n\to \infty}\prod_{k=0}^n b^*_k=0.
\end{align}
With the similar techniques as in the proof of Theorem~\ref{theo-IG-and-I} and~\ref{theo-IG-and-IM}, we can show that \begin{align}
	&\lim\limits_{n\to \infty}\prod_{k=0}^n a^*_k=\lim\limits_{n\to \infty}\prod_{k=0}^n \dfrac{\beta_1-(\beta_1-\alpha_1)a_k}{1-(1+\alpha_1)a_k}=0\Leftrightarrow \displaystyle\sum_{n \in \mathbb{N}}\left(\dfrac{1-\beta_1}{1-\beta_1+2\alpha_1}-a_n\right)=+\infty,\label{conv-2.3.1}\\
	&\lim\limits_{n\to \infty}\prod_{k=0}^n b^*_k=\lim\limits_{n\to \infty}\prod_{k=0}^n \dfrac{\beta_2-(\beta_2-\alpha_2)b_k}{\kappa_1-(\kappa_1+\alpha_2)b_k}=0\Leftrightarrow \displaystyle\sum_{n \in \mathbb{N}}\left(\dfrac{\kappa_1-\beta_2}{\kappa_1-\beta_2+2\alpha_2}-b_n\right)=+\infty.\label{conv-2.3.2}
\end{align}
Taking into account all estimates in~\eqref{est-2.3}-\eqref{conv-2.3.2}, we conclude that  
\begin{align*}
	\lim\limits_{n\to \infty} \mathcal{R}(x_n^G,x_n^{IG},x^*) =0,
\end{align*}
if the condition in~\eqref{ser} is valid. As a result, the iteration (G) converges faster than (IG).
\end{proof}

\section*{Conflict of Interest}
The authors declared that they have no conflict of interest.

\section*{Declarations}
Data sharing not applicable to this article as no datasets were generated or analysed during the current study.

\end{document}